\documentclass[11pt]{amsart}

\usepackage[cp1250]{inputenc}
\usepackage{color,amssymb}
\usepackage{enumerate}
\usepackage{graphicx}
\usepackage{pgfplots}

\pgfplotsset{compat = newest}

\input xy
\xyoption{all}

\newcommand{\F}{\mathcal F}

\newcommand{\I}{\mathcal I}

\newcommand{\N}{\mathbb N}
\newcommand{\R}{\mathbb R}

\newcommand{\C}{\mathbb C}
\newcommand{\K}{\mathcal K}

\newcommand{\Lim}{\operatorname{Lt}}
\newcommand{\Li}{\operatorname{Li}}
\newcommand{\Ls}{\operatorname{Ls}}

\newcommand{\Fix}{\operatorname{Fix}}

\renewcommand{\phi}{\varphi}
\newcommand{\w}{\omega}

\newcommand{\im}{\mathrm{im}}
\renewcommand{\int}{\mathrm{int}}

\newtheorem{theorem}{Theorem}[section]
\newtheorem*{theorem*}{Main theorem} 

\newtheorem{lemma}[theorem]{Lemma}

\newtheorem{corollary}[theorem]{Corollary}

\theoremstyle{definition}
\newtheorem{definition}[theorem]{Definition}
\newtheorem{remark}[theorem]{Remark}
\newtheorem{example}[theorem]{Example}

\title{Pointwise attractors which are not strict}

\author{Magdalena Nowak}
\address{M.Nowak: Mathematics Department, Jan Kochanowski University in Kielce, ul. Uniwersytecka 7, 25-406 Kielce, Poland}
\email{magdalena.nowak805@gmail.com}


\begin{document}
\begin{abstract}
We deal with the finite family $\F$ of continuous maps on the Hausdorff space $X$. A~nonempty compact subset $A$ of such space is called a strict attractor if it has an open neighborhood $U$ such that $A=\lim_{n\to\infty}\F^n(S)$ for every nonempty compact $S\subset U$.  Every strict attractor is a pointwise attractor, which means that the set $\{x\in X ; \lim_{n\to\infty}\F^n(x)=A\}$ contains $A$ in its interior.  

We present a class of examples of pointwise attractors -- from the finite set to the Sierpiński carpet -- which are not strict when we add to the system one nonexpansive map.
\end{abstract}
\maketitle

\section{Introduction}\label{intro}

Consider the following general version of an iterated function system.

\begin{definition}
	A pair $(X,\F)$ is called an \textit{iterated function system} (short IFS) if $X$ is a Hausdorff space and $\F$ is a finite family of continuous maps $X\to X$. By $\K(X)$ we denote the hyperspace of every nonempty compact subsets of $X$ with the Vietoris topology. Define also a \textit{Barnsley-Hutchinson} operator $\F\colon\K(X)\to\K(X)$ associated with the IFS $(X,\F)$ such that for every $S\in\K(X)$
	$$\F(S)=\bigcup_{f\in\F}f(S).$$
	We also use this formula for arbitrary $S\subset X$ and, for singletons we will write $\F^n(x)$ instead of $\F^n(\{x\})$.
\end{definition}
Following by the paper \cite{BLR} we define several two version of attractors for a given IFS $(X,\F)$.

\begin{definition}
	A \textit{strict attractor} of IFS $(X,\F)$ is a set $A\in\K(X)$ which has an open neighborhood $U$, such that $A=\lim_{n\to\infty}\F^n(S)$ for every $S\in\K(U)$. The maximal open set $U$ with the above property is called the \textit{basin} of the attractor $A$. 
\end{definition}

In the classical theory of IFS, where $\F$ is a family of contractions on the complete metric space $X$, due to the Hutchinson theorem, the whole space $X$ is the basin of the unique attractor of the given IFS. For example the singleton $\{0\}$ is a strict attractor of IFS $(\R,\{\frac{x}2\})$, the unit interval $[0,1]$ is a strict attractor of $(\R,\{\frac{x}2, \frac{x+1}2\})$ and the unit square $[0,1]^2$  is a strict attractor of $(\C,\{\frac{x}2, \frac{x+1}2, \frac{x+i}2, \frac{x+1+i}2\})$.

\begin{definition}
	A \textit{pointwise attractor} of IFS $(X,\F)$ is a set $A\in\K(X)$ satisfying $A\subset \int~B_p(A,\F)$, where 
	$$B_p(A,\F)=\{x\in X ; \lim_{n\to\infty}\F^n(x)=A\}$$
	is called a \textit{pointwise basin} of the attractor $A$.
\end{definition} 

Both attractors defined above are fixed points of operator $\F$ and their basins are open (see \cite{Bar}). Moreover, the space homeomorphic to such set remains an attractor of the corresponding IFS. It is easy to note that every strict attractor is pointwise one. Barnsley, Leśniak and Rypka shown in \cite{BLR} that the converse implication is true if $\F$ contains only nonexpansive maps on the metric space, but it is not true in general. The first counterexample, given by D. Kwietniak, is the following.
 
\begin{example}\label{ex_kwietniak}
	Let $S^1=\{z\in\C; |z|=1\}$ be the unit circle which may be projected to the set $\hat{\R}=\R\cup\{\infty\}$. Consider $\hat{\R}$ with an Euclidean metric on the circle and define a continuous map $\phi\colon \hat{\R}\to \hat{\R}$ such that $\phi(x)=x+1$ for $x\neq \infty$ and $\phi(\infty)=\infty$. Note that the singleton $\{\infty\}$ 
	is a pointwise attractor of $(\hat{\R},\{\phi\})$ 
	and $B_p(\{\infty\},\{\phi\})=\hat{\R}$. However, it is not a strict attractor of such IFS, cause $\{\infty\}\neq \lim_{n\to\infty}\phi^n(K)$ for every compact set $K=\{\infty\}\cup(-\infty,x]$.
\end{example}

In this paper we describe a wider class of counterexamples like that. In most of them only one map in the IFS is not nonexpansive. 

\section{Preliminaries and notation}

Consider the space $\K(X)$ for a given Hausdorff space $X$. We endow it with the Vietoris topology $\tau_V$ generated by the basis of sets
$$\langle U_1,\dots,U_n\rangle := \{A\in\K(X); A\subset\bigcup_{i=1}^nU_i \text{ and }A\cap U_i\neq\emptyset \text{ for every }i=1,\dots,n\}$$
where $U_1,\dots, U_n$ are open subsets of $X$. In this topology $\K(X)$ is a Hausdorff space (see \cite{Top}).

If $X$ is compact, then the convergence in the Vietoris topology in $\K(X)$ is equivalent with the Kuratowski convergence (so-called L-convergence), cf. \cite{Nad}. We say that a set $A\subset X$ \textit{is a Kuratowski limit (or topological limit)} of the sequence $(A_n)_{n\in\N}$ of subsets of $X$ if $A=\Li A_n=\Ls A_n$ where:
\begin{itemize}
	\item $\Li A_n = \{x\in X;$ for every open neighborhood $U$ of $x, A_n\cap U\neq\emptyset$ for almost every $n\in\N \}$
	\item $\Ls A_n = \{x\in X;$ for every open neighborhood $U$ of $x, A_n\cap U\neq\emptyset$ for infinitely many $n\in\N \}.$
\end{itemize}
We denote this by $A=\Lim A_n$.

Note that $x\in\Ls A_n$ iff there exists $(a_n)_{n\in\N}$ such that $a_n\in A_n$ and its subsequence $(a_{n_k})_{k\in\N}$ such that $x= \lim_{k\to\infty}a_{n_k}$. The following properties can be easily proved or found in \cite{Nad, Top}.
	\begin{remark}\label{remlim}
		For every sequences $(A_n)_{n\in\N}$, $(B_n)_{n\in\N}$ and $(C_n)_{n\in\N}$ of subsets of topological space $X$ 
	\begin{enumerate}[(i)]
		\item \label{rem_reg} if $X$ is a regular Hausdorff space then $$A=\lim_{n\to\infty}A_n \text{ in the Vietoris topology }\Rightarrow A=\Lim A_n$$
		\item if $X$ is a compact Hausdorff space then $$A=\lim_{n\to\infty}A_n \text{ in the Vietoris topology }\Leftrightarrow A=\Lim A_n$$
		\item \label{rem_compsup} if $D\in\K(X)$ and $A_n\in\K(D)$ for every $n$ then $$A=\lim_{n\to\infty}A_n \text{ in the Vietoris topology }\Leftrightarrow A=\Lim A_n$$
		\item \label{rem_union} if there exists $\lim_{n\to\infty} A_n$ and $\lim_{n\to\infty} B_n$ in the Vietoris topology then $$\lim_{n\to\infty}(A_n\cup B_n)=\lim_{n\to\infty} A_n \cup \lim_{n\to\infty} B_n$$
		
		
		\item\label{rem_subset} if $A_n\subset B_n$ for every $n$ then $\Li A_n \subset \Li B_n$ and $\Ls A_n \subset \Ls B_n$
		
		\item \label{rem_trzyciagi} if $A_n\subset B_n\subset C_n$ and $\lim_{n\to\infty} A_n=\lim_{n\to\infty} C_n$ then $\lim_{n\to\infty} B_n = \lim_{n\to\infty} A_n$.
	\end{enumerate}
\end{remark}



\section{Attracting maps with a local repellor}\label{sec_pf}

Let $\phi:X\to X$ be a continuous map on the Hausdorff space $X$. 
By $\Fix(\phi)$ we denote the set of fixed points of the function $\phi$. It is a closed set by continuity of $\phi$. 

\begin{definition}
	We say that a point $a\in\Fix(\phi)$ is a \textit{local repellor of} $\phi$ if it has a sequence $(x_n)_{n\in\N}$ in $X$ converges to the point $a$, called a \textit{witnessing sequence}, such that $x_0\notin\Fix(\phi)$ and $\phi(x_{n+1})=x_n$ for every $n\in\N$.
\end{definition}
Note that, in fact, no point in the witnessing sequence is a fixed point of $\phi$. 

\begin{theorem}\label{thm_rep}
	If for a set $A\subset X$, the map $\phi:X\to X$ has a local repellor $a\in A$ with the witnessing sequence $(x_n)_{n\in\N}$ such that $x_0\notin A$, then $A$ cannot be a strict attractor for the IFS containing $\phi$.
\end{theorem}

\begin{proof}
	By the fact that $a=\lim_{n\to\infty}x_n$, for every open $U\supset A$ there exists $n_0\in\N$ such that $x_n\in U$ for $n\geq n_0$. Then the set $K:=\{a\}\cup\bigcup_{n\geq n_0}\{x_n\}$ is a compact subset of $U$. If the map $\phi$ is contained in some finite family $\F$ of continuous maps $X\to X$, then 
	$$x_0=\phi^n(x_n)\subset \phi^n(K)\subset \F^n(K)$$ 
	for every $n\geq n_0$.
	Therefore $A\neq\lim_{n\to\infty}\F^n(K)$, otherwise
	$x_0$ must be an element of $A$.
\end{proof}	

As a consequence of that theorem 
we immediately obtain
\begin{corollary}\label{col_rep}
	If the map $\phi$ has a local repellor, then the set $\Fix(\phi)$ is not a strict attractor of any IFS containing $\phi$.
\end{corollary}

The map $\phi$ from Kwietniak's Example \ref{ex_kwietniak} satisfies the above corollary cause the only fixed point $\infty$ is also a local repellor of $\phi$ with a witnessing sequence $(-n)_{n\in\N}$. This function is also obeys the following definition which guarantees, in that case, that $\{\infty\}$ is a pointwise attractor.

\begin{definition}\label{def_pf}
	A continuous map $\phi\colon X\to X$ on a Hausdorff space $X$ is called an \textit{attracting map} when for every $x\in X$ there exists a limit $\lim_{n\to\infty}\phi^n(x)\in \Fix(\phi)$. 
\end{definition}

 We will use attracting maps to preserve a pointwise attractivity in the following way.
\begin{theorem}
	If $\phi$ is an attracting map on $X$ and $\Fix(\phi)$ is a pointwise attractor of IFS $(X,W)$ such that $W(X)\subset\Fix(\phi)$, then the set $\Fix(\phi)$ is a pointwise attractor of IFS $(X,W\cup\{\phi\})$.
\end{theorem}

\begin{proof}
	Let $\F:=W\cup\{\phi\}$. We obtain the assertion by the fact that   
	$$W^n(x)\subset \F^n(x) \subset \Fix(\phi)\cup\{\phi^n(x)\}$$
	for every $x\in B_p(\Fix(\phi),W)$ and $n\in\N$, and by the squeezing argument (cf. Remark \ref{remlim} (\ref{rem_trzyciagi})).
\end{proof}

These results reveal the importance of attracting maps with local repellor, so called \textit{ALR-maps}.

\begin{corollary}\label{col_main}
	If $\phi$ is an attracting map with local repellor (ALR-map) such that $\Fix(\phi)$ is a pointwise attractor of IFS $(X,W)$ and $W(X)\subset\Fix(\phi)$, then the set $\Fix(\phi)$ is a pointwise but not strict attractor of IFS $(X,W\cup\{\phi\})$.
\end{corollary}

\subsection{Examples of ALR-maps}
Let us present several examples of attracting maps with local repellor on $\R$, the unit circle $S^1$ and the unit ball in  $\C$.

Maps $x^2$ and $\sqrt x$ are simple examples of ALR-maps on the unit interval. In the following lemma we characterize a wider class of attracting maps with local repellor on the closed intervals in $\R$.

\begin{lemma}\label{pfonR}
	For a given interval $[a,b]\subset\R$, each continuous map $\phi\colon[a,b]\to[a,b]$ where $\Fix(\phi)=\{a,b\}$, is an ALR-map on $[a,b]$.
\end{lemma}
\begin{proof} Note that if $a$ and $b$ are the only fixed points of $\phi$, then $\phi(x)<x$ [or $\phi(x)>x$] for every $x\in(a,b)$. 
	
	For proving attractivity of $\phi$ observe that the sequence $(\phi^n(x))_{n\in\N}$ starting from arbitrary $x\in[a,b]$ is 
	decreasing [respectively increasing] 
	or constant for $n$ big enough. This means that there exist $g=\lim_{n\to\infty}\phi^n(x)$ which has to be a fixed point of $\phi$,	
	so $g\in\{a,b\}$. 
	
	Next, we show that $b$ is a local repellor if $\phi(x)<x$ for every $x\in(a,b)$. Indeed, we can construct a witnessing sequence $(x_n)_{n\in\N}$ starting from the arbitrary $x_0\in(a,b)$. If $x_n\in (a,b)$, then
	$$x_n\in[x_n,b]\subset[\phi(x_n),\phi(b)]\subset\phi([x_n,b]).$$ Therefore, for any natural $n$ we can find $x_{n+1}\in[x_n,b]$ such that $\phi(x_{n+1})=x_n$. Note that $x_n\neq b$ for every $n\in\N$. Otherwise for some $N\in\N$ we would have $x_0=\phi^N(b)=b$ which contradicts $x_0\in(a,b)$. 
	Thus we construct increasing sequence $(x_n)_{n\in\N}$ in $[x_0,b)$ which should have a limit in the set of fixed points of $\phi$. Hence $b=\lim_{n\to\infty}x_{n}$.\\
Analogously we can proof that, if $\phi(x)>x$ for every $x\in(a,b)$, then $a$ is a local repellor of $\phi$.
\end{proof}

Therefore we have the ensuing
 
\begin{example}\label{ex_A}
	On the arbitrary interval $[a,b]$ in $\R$ the following functions are attracting maps with a local repellor:
	$$\frac{(x-a)^2}{b-a}+a \quad\text{ or } \quad\sqrt{(x-a)(b-a)}+a.$$ 
\end{example}

In the Figure \ref{wykresy4} we present graphs of several ALR-maps on $[a,b]$ and diagrams of their dynamics: black points $a$ and $b$ are fixed points and
other elements from the interval are attracted to the fixed points according to the arrows.

\begin{figure}[h]
	
	\includegraphics[scale=0.6]{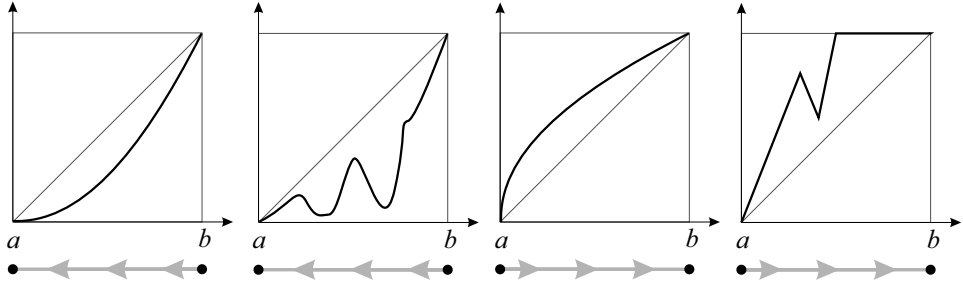}
	
	\caption{Examples of ALR-maps on the interval $[a,b]$.}\label{wykresy4}
\end{figure}

The following two lemmas will help us construct other ALR-maps.
 
\begin{lemma}\label{union}
	Assume that, for every index $i\in I$, the function $\phi_i$ is an attracting map on the Hausdorff space $X_i$ and
	\begin{enumerate}[(i)]
		\item $\phi_{i_0}$ has a local repellor for some $i_0\in I$
		\item\label{as_fix} $X_i\cap X_j\subset \Fix(\phi_i)\cap\Fix(\phi_j)$ for distinct $i,j\in I$.
	\end{enumerate}
	Then the union of maps $\Phi=\bigcup_{i\in I}\phi_i$ (the function $\Phi\colon \bigcup_{i\in I}X_i\to \bigcup_{i\in I}X_i$ such that
	$\Phi|_{X_i}= \phi_i$ for every $i\in I$) is an ALR-map and $\Fix(\Phi)=\bigcup_{i\in I}\Fix(\phi_i)$.  
\end{lemma}	

\begin{proof}
	The map $\Phi$ is well defined, cause for $x\in X_i\cap X_j$ we have $\Phi(x)=\phi_i(x)=\phi_j(x)=x$. It is easy to show that $\Fix(\Phi)=\bigcup_{i\in I}\Fix(\phi_i)$ and $\Phi$ is an attracting map.
	
	Note also that the local repellor $a$ of the map $\phi_{i_0}$ is also a local repellor of $\Phi$. Indeed, the point $a\in\Fix(\phi_{i_0})\subset \Fix(\Phi)$ is a limit of the witnessing sequence $(x_n)_{n\in\N}$ such that $x_n\in X_{i_0}\setminus \Fix(\phi_{i_0})$ for every $n\in\N$, and, by the condition (\ref{as_fix}), $x_n\notin X_j$ for any $j\neq i_0$. Thus $x_0\notin\Fix(\Phi)$ and $\Phi(x_{n+1})=\phi_{i_0}(x_{n+1})=x_n$.
\end{proof}

Now we can construct an ALR-map on the arbitrary subset of $\R$ containing a closed interval. It is enough to take a union of the ALR-map (from Lemma \ref{pfonR}) on this interval and the identity map (which is always an attracting map) on the rest of the space. See Figure \ref{schematy2}. 

\begin{figure}[h]
	\includegraphics[scale=0.6]{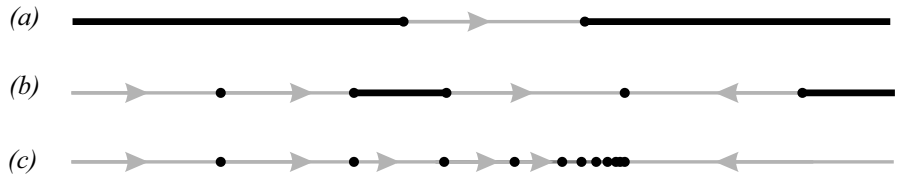}
	\caption{Diagrams of sample ALR-maps on the real line. Black areas indicate the action of the identity map.}\label{schematy2}
\end{figure}

The proof of next result is standard and left to the reader.

\begin{lemma}\label{hom}
	For an ALR-map $\phi$ on the Hausdorff space $X$ and for homeomorphism $h\colon X\to Y$, the composition $h\circ\phi\circ h^{-1}$ is also an ALR-map on $Y$.   
\end{lemma}

This allows us to generate an ALR-map, among others, on every arc (see Fig. \ref{krzywe}). 

\begin{figure}[h]
	\includegraphics[scale=0.6]{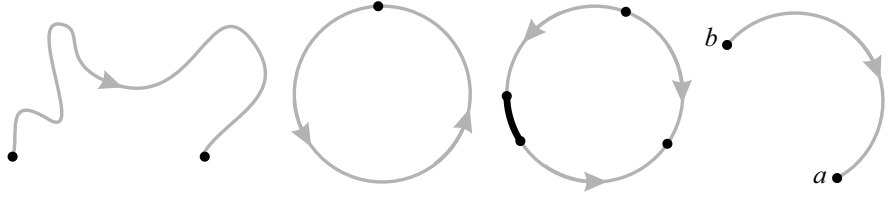}
	\caption{Diagrams of sample ALR-maps on curves.}\label{krzywe}
\end{figure}

\begin{example}\label{pfarc}
	Using the homeomorphism $h\colon[0,2\pi)\to S^1$ of the formula $h(x)=e^{ix}$, we can construct an ALR-map on the arbitrary 
	arc with endpoints $a,b\in S^1$, like in the Figure \ref{krzywe}. Suppose that  $\alpha=arg(a), \beta=arg(b)$ are in $[0,2\pi)$. Then we have the following example of the attracting map $\phi$ with a local repellor $b$:
	$$\phi(x) = \exp\big(i\big(\frac{(\arg x-\alpha)^2}{\beta-\alpha}+\alpha\big)\big),$$
	where "$-$" is an operation modulo $2\pi$ and for $\alpha=\beta$ we take $\beta-\alpha=2\pi$. 
\end{example}
The similar maps are proposed in 2020 by Fitzsimmons and Kunze in \cite{FK}. 

\begin{example}\label{ex:circ}
	We can also construct an ALR-map on the closed unit ball $B(0,1)\subset\C$. First define two continuous maps $a, b\colon B(0,1)\to S^1$ such that for every $x+iy\in B(0,1)$
	$$a(x+iy)=x-i\sqrt{1-x^2}\in S^1$$
	$$b(x+iy)=x+i\sqrt{1-x^2}\in S^1.$$ 
	Thus, for each $z\in B(0,1)$ points $a(z), z, b(z)$ are collinear. Then, the bijection 	
	$$\phi(z)=\Big|\frac{z-a(z)}{b(z)-a(z)}\Big|(z-a(z))+a(z)$$
	is an ALR-map on $B(0,1)$, constant on the circle $S^1$ (see the first diagram in the Figure \ref{fig:disc_tri_sqr}).
\end{example}

\begin{figure}[h]
	\includegraphics[scale=0.65]{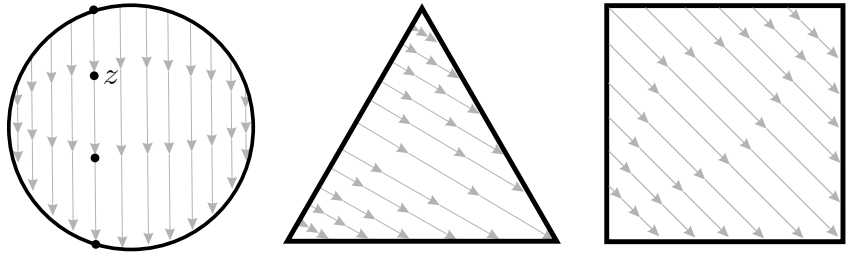}
	\caption{Diagrams of ARL-maps with fixed points on the boundary of the domain.}\label{fig:disc_tri_sqr}
\end{figure}

Indeed, for every $z\in B(0,1)\setminus S^1$ 
\begin{itemize}
	\item the point $\phi(z)$ lies in the segment between points $z$ and $a(z)$, so $|\phi(z)-a(z)|<|z-a(z)|$
	\item $a(\phi(z))=a(z)$ and $b(\phi(z))=b(z)$
	\item $\phi^n(z)=\big|\frac{z-a(z)}{b(z)-a(z)}\big|^{2n-1}(z-a(z))+a(z)$ for every $n\in\N^+$, so $\lim_{n\to\infty}\phi^n(z)= a(z)$
	\item for every $z\notin S^1$ the point $b(z)$ is a local repellor with a witnessing sequence $(\phi^{-n}(z))_{n\in\N}$.
\end{itemize}

Thanks to the Lemma \ref{hom} we can obtain an ALR-map on every set homeomorphic to the closed unit ball, like in the Figure \ref{fig:disc_tri_sqr}.

\section{Examples of pointwise and nonstrict attractors}

	Still let $\phi$ be a continuous function on the Hausdorff space $X$. Looking for maps satisfying the assumptions of Corollary \ref{col_main} we noted the following
	
\begin{lemma}\label{lem_retr}
	Let $A=\bigcup_{k=1}^m A_k$ for some natural $m\geq 1$ and for every $k=1,\dots,m$, the set $A_k$ is a retract of $X$ (by the retraction $r_k\colon X\to A_k$) and a pointwise attractor of IFS $(X,W_k)$.
	Then for $\F=\bigcup_{k=1}^m\{w\circ r_k; w\in W_k\}$ the set $A$ is a pointwise attractor for IFS $(X,\F)$ and $\F(X)= A$. 
\end{lemma}	

\begin{proof}
	Define $\tilde{W}_k:=\{w\circ r_k; w\in W_k\}$ for every $k=1,\dots,m$ and $\F:=\bigcup_{k=1}^m\tilde{W}_k$. Note that  $$\tilde{W}_k(X)=W_k(r_k(X))=W_k(A_k)=A_k,$$
	for each $k$, so also $\F(X)=A$. We show that $A$ is a pointwise attractor of $(X,\F)$ and $B_p(A,\F)=X$. Take $x\in X$ and note that  
	$\bigcup_{k=1}^m\tilde{W}^n_k(x)\subset \F^n(x) \subset A$ for every $n\in\N$.
	Moreover, 
	$$\lim_{n\to\infty}\bigcup_{k=1}^m\tilde{W}^n_k(x)=\bigcup_{k=1}^m\lim_{n\to\infty}\tilde{W}^n_k(x)=\bigcup_{k=1}^m\lim_{n\to\infty}{W}^n_k(r_k(x))=\bigcup_{k=1}^m A_k=A.$$ 
	This implies that $\lim_{n\to\infty}\F^n(x)=A$, so $A$ is a pointwise attractor of $(X,\F)$.
\end{proof}	

Consequently, by the Corollary \ref{col_main}  we have the following main theorem.
\begin{theorem}\label{main_thm}
	Let $\phi\colon X\to X$ be an attracting map with local repellor such that $\Fix(\phi)$ is a finite union of sets each of which is a retract of $X$ and a pointwise attractor on $X$. Then there exists a family $\F$ containing $\phi$, such that $\Fix(\phi)$ is a pointwise attractor of IFS $(X,\F)$ with a pointwise basin $X$ but it is not a strict attractor of $(X,\F)$.
\end{theorem}	

	


	As a corollary of the above theorem, using ALR-maps on intervals and arcs we can construct several examples of pointwise attractors which are not strict.
	
	\begin{example}\label{ex_R}
		If $A$ is a finite union of closed arcs or singletons in $S^1$, then there exists an IFS $(S^1,\F)$ for which $A$ is not a strict attractor but it is a pointwise attractor.
	
	Let $A=\bigcup_{k=1}^m A_k$ is a nonempty, finite union of singletons and closed arcs $A_k$. We can assume that the above union is disjoint. Each set $A_k$ is a pointwise (and strict) attractor and a retract of $S^1$. 
	The set $S^1\setminus A$ is a disjoint union of $m$ open arcs $I_k = (a_k, b_k)$. For every $k=1,\dots,m$ we can find an ALR-map $\phi_k$ on $\overline{I_k}$ such that $\Fix(\phi_k)=\{a_k, b_k\}$ (see Example \ref{pfarc}). Then, by the Lemma \ref{union}, the union $\phi=\bigcup_{k=1}^m\phi_k \cup id_A$ is an ALR-map on $S^1$ (see Figure \ref{okrag}). Thus, by the Theorem \ref{main_thm}, the set $A=\Fix(\phi)$ is a pointwise but not strict attractor for some IFS.
	\end{example}
	\begin{remark}	
		For finite set $A\subset S^1$ we can obtain the same result with the family $\F=	\{\phi,\psi\}$ contains only two maps: $\phi$ constructed in Example \ref{ex_R} and the map $\psi=\phi\circ w$ (or $w\circ\phi$) where $w$ is a continuous, piecewise linear function on $S^1$ such that $w|_{I_k}\colon I_k\to I_{(k+1)mod{(m)}}$ for every $k=1,\dots,m$ preserves the ratio of arc lengths. See Figure \ref{okrag}.
	\end{remark}	
	
	\begin{figure}[h]
		\includegraphics[scale=0.5]{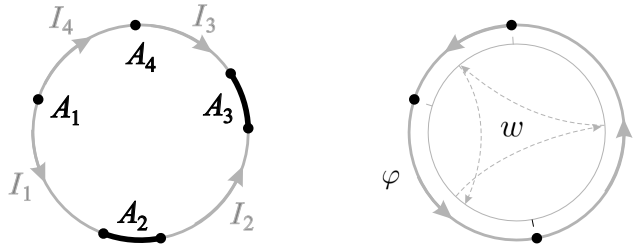}
		\caption{Pointwise, nonstrict attractors on $\hat{\R}$.}\label{okrag}
	\end{figure}
	
	\begin{example}\label{ex_RR}
		If $A$ is a finite union of at least two closed intervals or singletons in $\R$, then there exists an IFS for which $A$ is not a strict attractor but it is a pointwise attractor.
	\end{example}

	Take the notation from the previous example. Note that now the complement $\R\setminus A$ contains $m+1$ intervals $I_k=(a_k,b_k)$ for $k=0,\dots,m$. Assume that $-\infty=a_0< b_0\leq a_1< b_1\leq ...\leq a_m<b_m=+\infty$. For every $k=1,...,m-1$ take the ALR-map $\phi_k$ on $[a_k,b_k]$ like in Example \ref{ex_A}. Then we define $\phi\colon\R\to\R$ such that
	$$\phi(x)=
	\begin{cases}
		b_0 &\text{for }x< b_0\\
		\phi_k(x) &\text{for }x\in (a_k,b_k) \text{ and }k=1,\dots,m-1\\
		x & \text{for }x\in A \\
		a_m &\text{for }x> a_m.
	\end{cases}$$
	
	By the Lemma \ref{union} $\phi$ is an ALR-map on $\R$ and $A=\Fix(\phi)$ is a pointwise but not strict attractor (see Figure \ref{prosta}).
		
\begin{remark}
	In fact it is enough to take the following map
	
	$$\phi(x)=
	\begin{cases}
		a_1 &\text{for }x< a_1\\
		\phi_1(x) &\text{for }x\in [a_1,b_1]\\
		b_1 &\text{for }x> b_1.
	\end{cases}$$
	
	It is an ALR-map on $\R$ but it may not satisfy the assumption of Theorem \ref{main_thm} cause $\Fix(\phi)=\{a_1,b_1\}$. Nevertheless, the proof of pointwise attractivity works for that map due to the fact that the image $\im(\phi\circ w)\subset \Fix(\phi)\subset A$ for every map $w\in\F\setminus\{\phi\}$.
\end{remark}

	\begin{figure}[h]
		\includegraphics[scale=0.5]{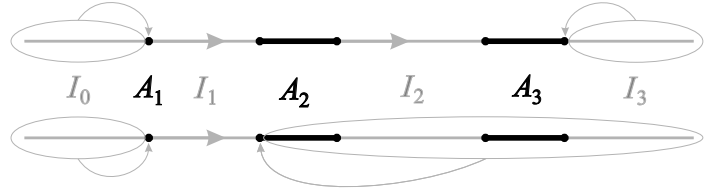}
		\caption{A pointwise, nonstrict attractor on the real line and two variants of the map $\phi$.}\label{prosta}
	\end{figure}

\begin{example}
	The finite union of singletons, curves and sets homeomorphic to the unit square in $\C$ or the sphere $S^2$ is a pointwise, nonstrict attractor of the IFS contains the ALR-map $\phi$ which diagrams are presented in the Figure \ref{sfera}.
\end{example}
	\begin{figure}[h]
		\includegraphics[scale=0.7]{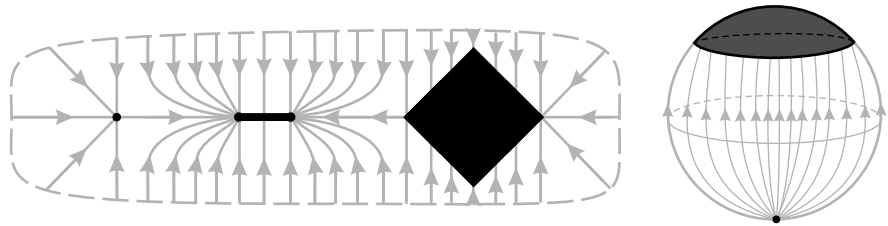}
		\caption{Pointwise, nonstrict attractors on $\C$ and $\hat{\C}$.}\label{sfera}
	\end{figure}

\section{Classical fractals as no strict attractors}

The results obtained in the previous section cannot be applied to many classical IFS-attractors that are not a finite union of retracts. Therefore, we had to develop some new approach were the added function $\phi$ must cooperate with other maps from IFS (for example, compositions of that maps must be commutative on a certain set).

Recall that, for $W, V$ two families of maps, we denote by $W\circ V=\{w\circ v; w\in W, v\in V\}$ the set of all compositions of maps from both sets. Inductively define $W^0=\{id\}$, $W^{n+1}=W^n\circ W$ and $W^{<\w}=\bigcup_{n\in\N} W^n$. 
In order to avoid notation with multiple indexes, we will use the following symbols for subsequences and their limits. For a given sequence $(x_n)_{n\in\N}$ and an infinite $N\subset \N$, by $(x_n)_{n\in N}$ we denote the subsequence of $(x_n)_{n\in\N}$ and its limit (if exists) by
$\lim_{n\in N}x_n.$

\begin{theorem}\label{thm_cantor}
	If $A$ is a pointwise attractor of $(X,W)$ on the regular Hausdorff space $X$ and there exist a compact set $D$ such that 
	\begin{enumerate}[1.]
		\item $A\subset W(D)\subset D\subset B_p(A,W)$
		\item $D\setminus A$ is contained in the union of a family $\I$ of disjoint open sets such that for every $I\in\I$ the set $W_I:=\{w\in W^{<\omega}; w(D)\cap I\neq\emptyset\}$ is finite
		\item there exists an ALR-map $\phi\colon X\to X$ such that
		\begin{enumerate}[a)]
			\item $\Fix(\phi)= A$ 
			\item $\phi(X)\subset D$ and $\phi(I)= I$ for every $I\in\I$
			\item $\phi\circ w(x)=w\circ\phi(x)$ for every $x\in D$ and $w\in W$ (commutativity on $D$)
		\end{enumerate}
	\end{enumerate} 
then $A$ is a pointwise attractor but not a strict attractor for IFS $(X,W\cup\{\phi\})$.
\end{theorem}

\begin{proof}
	Denote $\F=\tilde{W}\cup\{\phi\}$ and observe that $\F(D)\subset D$. By the Corollary \ref{col_rep} we immediately have that $A$ is not a strict attractor of $(X,\F)$. Now we will show that $A$ is a pointwise attractor of $(X,\F)$ and $B_p(A,\F)\supset B_p(A,W)$. We divide this proof into several steps. \\
	\newline
	\textbf{STEP 1:} $A\subset \Li\F^n(x)$ for every $x\in B_p(A,W)$.
	
	We obtain this fact by Remark \ref{remlim} (\ref{rem_reg}) and (\ref{rem_subset}). Indeed
	$A=\lim_{n\to\infty} W^n(x)=\Li {W}^n(x) \subset \Li\F^n(x).$\\
	\newline
	\textbf{STEP 2:} $\Ls\F^n(x)\subset A$ for every $x\in B_p(A,W)$.
	
	Take $x\in B_p(A,W)$ and $y\in\Ls\F^n(x)$, which means that $y=\lim_{n\in N}f_n(x)$ for some $(f_n)_{n\in N}$ such that $f_n\in\F^{n}$ and for infinite set $N\subset\N$. Let us exclude two cases:\\
	\newline
	{Case 1.} $y\notin D$. Then, by regularity of $X$, we can find two disjoint open sets: $V\supset D$ and $U\ni y$. The set $V$ is also open neighborhood of the attractor $A$ so for $n\in N$ big enough $W^n(x)\subset V$ and
	$$f_n(x)\in \F^n(x)\subset W^n(x)\cup D \subset V,$$
	thus such points cannot be contained in the set $U$. This leads to a contradiction to $y=\lim_{n\in N}f_n(x)$.\\
	\newline
	{Case 2.} $y\in D\setminus A$. Then, for all but finite $n\in N$, in the composition $f_n$ the map $\phi$ appears at least once. Otherwise $f_n\in W^n$ for $n$ from some infinite set $M\subset N$ and 
	$$y= \lim_{n\in M}f_n(x)\in \lim_{n\in M}W^n(x)=A.$$
	This means that, by commutativity on $D\supset\phi(X)$, for $n\in N$ big enough we must have 
	$$f_n(x) = w_n \circ \phi^{n-k_n-s_n}\circ v_n(x)\in w_n(D),$$
	where $w_n\in W^{k_n}, v_n\in W^{s_n}$ and $0\leq k_n+s_n<n$ for some sequences of natural numbers $(k_n)_{n\in N}, (s_n)_{n\in N}$. 
	Since $y\in D\setminus A=\bigcup \I$, then $y\in I$ for some $I\in \I$ and also $f_n(x)\in I$ for $n\in N$ big enough. Hence, for every $n$ from some infinite set $N'\subset N$ we have $f_n(x)\in I \cap w_n(D)$, so also $w_n$ must be an element of the finite set $W_I$. Here we will use the so-called pigeonhole principle, which states that, if the expressions of a given sequence are in a finite union of sets, then one of these sets must contain its subsequence. Thus, there exists a continuous map $\textbf{w}\in W_I$ which is a composition from $W^K$ for some natural $K$, such that for infinitely many $n$ (say for $n$ from some infinite set $N''\subset N'$) 
	$$f_n(x) = \textbf{w}\circ\phi^{n-K-s_n}\circ v_n(x)=\phi^{n-K-s_n-1}\circ\textbf{w}\circ\phi\circ v_n(x),$$
	thanks of the commutativity on $D$. 
	
	From the regularity of $X$, take an open disjoint sets $U\ni y$ and $V\supset A\supset\textbf{w}(\phi(A))$. Then by the continuity of $\textbf{w}\circ\phi$ there exists $V'\supset A$ such that $\textbf{w}(\phi(V'))\subset V$. According to the fact that $A$ is a pointwise attractor of $(X,W)$, there exists $s_0\in\N$ such that for all but finite $s\geq s_0$ we have $W^s(x)\subset V'$ and also $\textbf{w}(\phi(W^s(x)))\subset V$. Thus, $\textbf{w}\circ\phi\circ v_n(x)\in U\cap I$ only for finitely many maps $v_n\in \bigcup_{s<s_0}W^s$. For those functions and $n\in N''$ points $f_n(x)=\phi^{n-K-s_n-1}\circ\textbf{w}\circ\phi\circ v_n(x)$ have a chance to be an elements of the set $I$. Then again, by the pigeonhole principle, there exists a constant $S<s_0$ and a map $\textbf{v}\in W^S$ such that $f_n(x)=\phi^{n-K-S-1}\circ\textbf{w}\circ\phi\circ\textbf{v}(x)\in I$ for $n$ from infinite set $N'''\subset N''$. This means that, taking $z:=\textbf{w}\circ\phi\circ\textbf{v}(x)$ we obtain
	$$y=\lim_{n\in N'''}f_n(x)=
	\lim_{n\in N'''}\phi^{n-K-S-1}(z) =
	\lim_{n\to \infty}\phi^{n}(z)
	\in\Fix(\phi)=A$$
	which is a contradiction. 
		
	This means that $y$ have to be an element of the set $A$, so $\Ls\F^n(x)\subset A$. Steps 1, 2 and the fact that $\Li\F^n(x)\subset\Ls\F^n(x)$ imply that $\Lim \F^n(x)=A$.\\ 
	\newline
	\textbf{STEP 3:} $A=\lim_{n\to\infty}\F^n(x)$.
	
	We obtain this fact by Remark \ref{remlim} (\ref{rem_compsup}). Indeed, for every $n\in\N$ the finite set $\F^n(x)$ is a subset of $D\cup\bigcup_{k\in\N} W^k(x)$, which is a compact set cause $\lim_{k\to\infty}W^k(x)= A\subset D$. This ends the proof.
\end{proof}	

Despite the long list of assumptions in the above theorem, there are several examples that satisfy them. It is known that the classical fractal like ternary Cantor set, Sierpiński triangle and Sierpiński carpet are strict and pointwise attractors. Our result says that, adding only one nonexpansive map to the IFS, we can spoil their strict attractivity.

\begin{example}
	The Cantor set in $\R$ or $\hat{\R}$ has an IFS for which it is a pointwise attractor but not a strict attractor.
\end{example}

Let $A$ be the ternary Cantor set in the unit interval $D:=[0,1]$. Take the standard IFS for the Cantor set: $W=\{\frac{x}3, \frac{x+2}3\}$ on $\R$ (or their continuous extension to $\hat{\R}$). The pointwise basin $B_p(A,W)=\R \supset D \supset W(D) \supset A$. The set $D\setminus A$ is a union of countable family $\I$ of disjoint open intervals.
Note that the set $W_I:=\{w\in W^{<\omega}; w([0,1])\cap I\neq\emptyset\}$ is finite for every $I\in\I$. Indeed for $I\in\I$ there exists a number $N$ such that  $W^N(D)\cap I=\emptyset$ for every $n\geq N$ cause $W^n(D)$ is the $n$-th step of the standard construction of the ternary Cantor set. 

Denote $I_0=(\frac13,\frac23)$ and observe that for each $I\in\I$ there exists a unique affine bijection  $w_I$ in $W_I$ such that $w_I(I_0)=I$. Moreover for every $w\in W^{<\w}$ the set $w(I)$ is always an element of $\I$, so 
\begin{align}
	\label{aff}
	w_{w(I)}=w\circ w_I,
\end{align}
cause $w\circ w_I(I_0)=w(I)$.

Let us define a map $\phi$ satisfying the last assumption of the Theorem \ref{thm_cantor}. Take an arbitrary ALR-map $\phi_0\colon \overline{I_0}\to \overline{I_0}$ such that $\Fix(\phi_0)=\{\frac13,\frac23\}$ (like in the Example \ref{ex_A}). Then
 $$\phi(x)=
 \begin{cases}
 x &\text{for }x\in A,\\
 w_I\circ \phi_0\circ w_{I}^{-1}(x) &\text{for }x\in I \text{ where }I\in\I,\\
  r(x) &	\text{for }x\in \R\setminus D~(\text{or }\hat{\R}\setminus D).
 \end{cases}$$
 where $r$ is an arbitrary retraction on $[0,1]$. It is easy to check that $\phi$ is an ARL-map and $A=\Fix(\phi)$. Moreover the image $\im\phi=D$ and $\phi(I)=I$ for every interval $I\in\I$. The commutativity of maps on $D$ goes form the equation (\ref{aff}).
 Indeed, for an element $x\in D=[0,1]$ and the arbitrary map $w\in W$ if $x\in A$ then $w(x)\in A$ so $\phi(w(x))=w(x)=w(\phi(x))$. If $x\in D\setminus A$ then $x\in I$ for some $I\in\I$ and
 $$w\circ\phi(x)=w\circ w_I\circ \phi_0\circ w_I^{-1}(x) = w_{w(I)}\circ\phi_0 \circ w_{w(I)}^{-1}\circ w(x)=\phi\circ w(x).$$
 
 All assumptions of the Theorem \ref{thm_cantor} are satisfied which means that the ternary Cantor set $A$ is a pointwise but not strict attractor for IFS $(\R,W\cup\{\phi\})$ and $(\hat{\R},W\cup\{\phi\})$.

 The similar construction can be made also for the Sierpiński triangle and the Sierpiński carpet. Then the set $D$ should be a convex hull of the fractal and $\phi_0$ is an arbitrary ALR-map on the closure of biggest "hole" $I_0$ (see Figure \ref{fraktaleR}).
 
 \begin{example}
 	The Sierpiński triangle and the Sierpiński carpet can be pointwise but no strict attractors for some IFS $(\C,\F)$ or  $(S^2,\F)$.
 \end{example}
 
 \begin{figure}[h]
 	\includegraphics[scale=0.7]{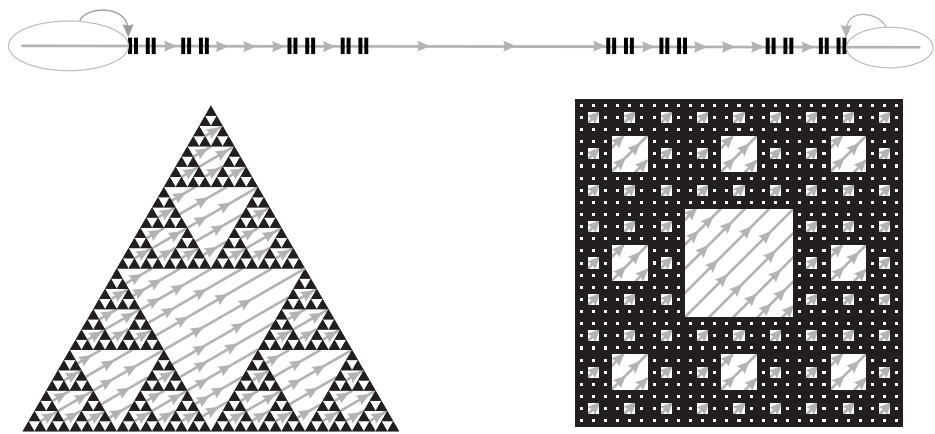}
 	\caption{Diagrams of ALR-maps on the classical fractals}\label{fraktaleR}
 \end{figure}

\begin{remark}
	The construction of $\phi$ in the examples above can be simplified such that
	$\phi(x)=\phi_0\circ r_0$ where $r_0$ is an arbitrary retraction to $\overline{I_0}$. Then $\phi$ is not satisfied the last assumption of Theorem \ref{thm_cantor} but it can also be shown that the fractal $A$ is still pointwise but not a strict attractor. The proof of that fact is similar to the proof of Theorem \ref{thm_cantor} so we decided to omit it from this article.
\end{remark}

\section*{Acknowledgements}

The author would like to thank Krzysztof Leśniak and Piotr Niemiec for many fruitful discussions which helped to write this paper.


\begin{thebibliography}{M}

\bibitem{Bar} P. G. Barrientos, F. H. Ghane, D. Malicet, A. Sarizadeh, \emph{On the chaos game of iterated function systems}. Topol. Methods Nonlinear Anal. \textbf{49}(1), 2017, 105–132.

\bibitem{BLR} M. F. Barnsley, K. Leśniak, M. Rypka, \emph{Chaos game for IFSs on topological spaces} Journal of Mathematical Analysis and Applications, \textbf{435}(2), 2016, 1458-1466.

\bibitem{FK} M. Fitzsimmons, H. Kunze \emph{Small and minimal attractors of an IFS} Commun Nonlinear Sci Numer Simul, \textbf{85}, 2020, 105227

\bibitem{Nad} A. Illanes, S. B. Nadler Jr., \emph{Hyperspaces. Fundamentals and Recent Advances}, M.Dekker, New York - Basel 1999.

\bibitem{Top} E. Michael, \emph{Topologies on spaces of subsets} Trans. Amer. Math. Soc. \textbf{71} 1951, 152-182.



\end{thebibliography}
\end{document}